\documentclass{article}
\usepackage[utf8]{inputenc}
\usepackage{amssymb}
\usepackage{amsmath, amssymb, amsthm}
\usepackage{url}
\newtheorem{conjecture}{Conjecture}
\usepackage{dsfont}
\usepackage{textcomp}
\usepackage{setspace}
\usepackage{float}
\usepackage{mathpazo}       
\usepackage[english]{babel}
\usepackage{graphicx}
\graphicspath{{./images/}}
\usepackage{amsmath}

\newtheorem{theorem}{\sc Theorem}[section]

\newtheorem{prop}[theorem]{\sc Proposition}
\newtheorem{cor}[theorem]{\sc Corollary}

\newtheorem{ex}[theorem]{\sc Example}

\usepackage[nottoc]{tocbibind}
\usepackage{cite}  

\begin{document}

\title{Proportion of Nilpotent Subgroups in Finite Groups and Their Properties}
\author{João Victor Monteiros de Andrade\thanks{Departamento de Computação, Universidade de Brasília. \texttt{jotandrade98@gmail.com}} 
\and Leonardo Santos\thanks{Departamento de Computação, Universidade de Brasília. \texttt{leonardo-7238@hotmail.com}}}

\maketitle

\section{Abstract}

\quad\,\, This work introduces and investigates the function \( J(G) = \frac{\text{Nil}(G)}{L(G)} \), where \( \text{Nil}(G) \) denotes the number of nilpotent subgroups and \( L(G) \) the total number of subgroups of a finite group \( G \). The function \( J(G) \), defined over the interval \((0,1]\), serves as a tool to analyze structural patterns in finite groups, particularly within non-nilpotent families such as supersolvable and dihedral groups. Analytical results demonstrate the product density
 of \( J(G) \) values in \((0,1]\), highlighting its distribution across products of dihedral groups.  Additionally, a probabilistic analysis was conducted, and based on extensive computational simulations, it was conjectured that the sample mean of \( J(G) \) values converges in distribution to the standard normal distribution, in accordance with the Central Limit Theorem, as the sample size increases. These findings expand the understanding of multiplicative functions in group theory, offering novel insights into the structural and probabilistic behavior of finite groups. 
 
\noindent{\textbf{Keywords:} Nilpotent subgroups; multiplicative functions; probabilistic analysis; dihedral groups; GAP}

\section{Introduction}

\quad \, The study of the structural properties of finite groups is a central theme in group theory, with special interest in understanding the distribution and influence of subgroups with specific characteristics. Among these properties, nilpotency stands out, widely studied for its direct relationship with the internal structure of groups and its applications in various algebraic contexts. Multiplicative functions defined based on subgroup characteristics, such as the degree of commutativity and the count of cyclic subgroups, have been explored to analyze asymptotic patterns and solubility criteria in finite groups [4], [6], [7].. 

In this work, we introduce the function $\mathfrak{J}(G)$, which expresses the ratio between the number of nilpotent subgroups and the total number of subgroups of a finite group \( G \), formally defined as:

\begin{equation}\label{eq.1}
   \mathfrak{J}(G) = \frac{\text{Nil}(G)}{L(G)}, 
\end{equation}

where \( \text{Nil}(G) \) represents the number of nilpotent subgroups of \( G \) and \( L(G) \) the total number of subgroups of \( G \). This function assumes values in the interval \((0, 1]\) and presents multiplicative behavior, allowing the analysis of structural patterns in direct products of groups with coprime orders.

The function $\mathfrak{J}(G)$ is particularly interesting when applied to families of non-nilpotent groups, such as supersolvable groups and, more specifically, dihedral groups. Through this approach, we explore the asymptotic behavior of $\mathfrak{J}(G)$ and demonstrate how this function can be used to characterize the nilpotency of groups and identify structural patterns in specific subclasses.

In addition, we develop a detailed analysis of the density of the values of $\mathfrak{J}(G)$ in the interval \((0,1]\), showing that these values are dense in this interval when considered as products of dihedral groups of specific orders. This result contributes to the understanding of the distribution of the function in different algebraic contexts.

Complementing the structural analysis, we perform a probabilistic investigation based on the Central Limit Theorem. We conjecture that the sample mean of the values of $\mathfrak{J}(G)$, obtained by random sampling of dihedral groups of increasing orders, converges in distribution to a standard normal \( N(0,1) \) when the sample size tends to infinity. This result demonstrates the probabilistic behavior of the function $\mathfrak{J}(G)$ and reinforces the robustness of its application in statistical analyses of properties of finite groups.

The present study extends the known results on multiplicative functions associated with finite groups, offering a new perspective for the analysis of structural and probabilistic properties. In particular, we highlight the density of values of $\mathfrak{J}(G)$ in the interval \((0,1]\) and the convergence to the standard normal distribution, as predicted by the Central Limit Theorem.

In the following sections, we will present the main properties of the function $\mathfrak{J}(G)$, demonstrate specific results for dihedral and dicyclic groups, and discuss possible future extensions and applications of this approach in group theory.

\subsection{Basic properties of $\mathfrak{J}$}
 Clearly this function is multiplicative, that is, if $gdc(|G_{1}|,...,|G_{n}|) = 1$, then

\begin{equation*}
    \mathfrak{J}(G_i \times ... \times G_n) = \prod_{i=1}^{n}\mathfrak{J}(G_{i}).
\end{equation*}

\begin{prop} \label{prop 1}
The group $G$ is nilpotent if and only if

\begin{equation*}
    \mathfrak{J}(G) = 1.
\end{equation*} 
\end{prop}

\begin{proof}
If $G$ is nilpotent every subgroup of $G$ will also be nipotent, this way, the numerator will coincide with the denominator and the result follows. The converse follows from the definition of the function since trivial subgroups are considered when counting subgroups.   
\end{proof}

From the Proposition \ref{prop 1} it follows that cyclic groups, abelians and finite p-groups will always return 1. In this case, it is convenient to use the function in families of non-nilpotent groups such as, for example, in supersoluble groups. One of the important subfamilies of this class is the family of finite dihedral groups.   

\begin{theorem}
Let \(D_{2n}\) be the dihedral group of order \(2n\) and \(\tau(n)\) be the number of positive divisors of \(n\). Then the total number of nilpotent subgroups of \(D_{2n}\) is given by an expression of the type
\[
\tau(n) \;+\; \sum_{\substack{2^r \mid n \\ r \geq 0}} \frac{n}{2^r}.
\]   
\end{theorem}

\begin{proof}
  The cyclic subgroups of \(D_{2n}\) are derived from the rotation elements \(r\). For each divisor \(d\) of \(n\), there is exactly one cyclic subgroup of order \(d\), denoted by \(\langle r^{n/d} \rangle\). Since cyclic subgroups are always nilpotent, the total number of nilpotent cyclic subgroups in \(D_{2n}\) is given by \(\tau(n)\), the number of divisors of \(n\). Subgroups that contain reflections can be of two types: Subgroups isomorphic to \(D_{2d}\), for some divisor \(d\) of \(n\): So that \(D_{2d}\) is nilpotent, it must be a \(2\)-group, which implies that \(2d\) is a power of 2. Thus, \(d = 2^r\), with \( 2^r \mid n \). For each   \( d = 2^r \), there are \( \frac{n}{d}\) subgroups isomorphic to \(D_{2d}\). Subgroups of order 2 generated by individual reflections: For each reflection \(sr^k\), where \(k \in \{0, 1, \dots, n-1 \} \), there is a cyclic subgroup of order 2 . Since there are \(n\) reflections in \(D_{2n}\), there are \(n\) nilpotent subgroups of order 2 of this type.
\end{proof}

\begin{prop}
Let $G = D_{2n}$ be a finite dihedral group of order $2n$. If $n = 3p$, where $p \geq 3$ is prime, then

\begin{equation*}
    \lim_{p \longrightarrow \infty } \mathfrak{J}(G) = \frac{3}{4}.
\end{equation*} 
\end{prop}

\begin{proof}
In fact, for a dihedral group, it is shown in [2]
   
   \begin{equation} \label{eq.2}
             L(D_{2n})  =  \tau(n) + \sigma(n)
     \end{equation}
     we can replace \ref{eq.2} em \ref{eq.1}

\begin{align*}
    \mathfrak{J}(D_{2n}) = \frac{Nil(D_{2n})}{\tau(n) + \sigma(n)} &= \frac{Nil(D_{2n})}{\tau(3p) + \sigma(3p)}\\ 
    &= \frac{n+4}{4 + 4 \Big(\frac{p^{2}-1}{p-1} \Big)}\\
    &= \frac{3p+4}{4 + 4( p + 1 )}\\
    &= \frac{3p}{4 + 4( p + 1 )} + \frac{4}{4 + 4( p + 1 )}.
\end{align*}

In this way there is

\begin{align*}
    \lim_{p \longrightarrow \infty }\mathfrak{J}(D_{2n}) = \lim_{p \longrightarrow \infty } \frac{3p}{4 + 4( p + 1 )} + \lim_{p \longrightarrow \infty } \frac{4}{4 + 4( p + 1 )} = \frac{3}{4}.
\end{align*}
\end{proof}

\begin{prop}
Let $G = D_{2n}$ be a finite dihedral group of order $2n$. If $n = 2\cdot3 \cdot p$, where $p \geq 5$ is prime, then

\begin{equation*}
    \lim_{p \longrightarrow \infty } \mathfrak{J}(G) = \frac{3}{4}.
\end{equation*} 
\end{prop}

\begin{proof}
    \begin{align*}
    \mathfrak{J}(D_{2n}) = \frac{Nil(D_{2n})}{\tau(n) + \sigma(n)} &= \frac{Nil(D_{2n})}{\tau(3p) + \sigma(3p)}\\ 
    &= \frac{8+9p}{8 + 12 \Big(\frac{p^{2}-1}{p-1} \Big)}\\
    &= \frac{8+9p}{8 + 12( p + 1 )}\\
    &= \frac{9p}{8 + 12( p + 1 )} + \frac{8}{8 + 12( p + 1 )}.
\end{align*}

Then it follows

\begin{align*}
    \lim_{p \longrightarrow \infty }\mathfrak{J}(D_{2n}) = \lim_{p \longrightarrow \infty } \frac{9p}{8 + 12( p + 1 )} + \lim_{p \longrightarrow \infty } \frac{8}{8 + 12( p + 1 )} = \frac{3}{4}.
\end{align*}

\end{proof}

\begin{prop}
Let $G = D_{2p}$ be a finite dihedral group of order $2p$. If $p$ is prime, then

\begin{equation*}
    \lim_{p \longrightarrow \infty } \mathfrak{J}(G) = 1.
\end{equation*} 
\end{prop}

\begin{proof}
    In fact,
    \begin{align*}
    \mathfrak{J}(D_{2p}) = \frac{Nil(D_{2p})}{\tau(p) + \sigma(p)} &= \frac{p +2}{\tau(p) + \sigma(p)}\\ 
    &= \frac{p+2}{2 + \Big(\frac{p^{2}-1}{p-1} \Big)}\\
    &= \frac{p+2}{2 + ( p + 1 )}\\
    &= \frac{p}{2 + ( p + 1 )} + \frac{2}{2 + ( p + 1 )}.
\end{align*}

In this way there is

\begin{align*}
    \lim_{p \longrightarrow \infty }\mathfrak{J}(D_{2p}) = \lim_{p \longrightarrow \infty } \frac{p}{2 + ( p + 1 )} + \lim_{p \longrightarrow \infty } \frac{2}{2 + ( p + 1 )} = 1.
\end{align*}
\end{proof}

\begin{prop}
Let $G = D_{2^{2}p}$ be a finite dihedral group of order $4p$. If $p$ is prime, then

\begin{equation*}
    \lim_{p \longrightarrow \infty } \mathfrak{J}(G) = 1.
\end{equation*} 
\end{prop}

\begin{proof}

     If  $|G| = 2^{2}p $, then  


\begin{align*}
     \mathfrak{J}(D_{2^{2}p}) = \frac{Nil(D_{2^{2}p})}{\tau(2p) + \sigma(2p)} &= \frac{ \Big(\frac{2^{2}p}{2} + 4 \Big) + p}{\tau(2p) + \sigma(2p)}\\
     &= \frac{ \Big(\frac{2^{2}p}{2} + 4 \Big) + p}{4 + \Big(\frac{2^{2}-1}{2-1} \Big) \Big(\frac{p^{2}-1}{p-1} \Big)}\\
     &= \frac{ \Big(\frac{2^{2}p}{2} + 4 \Big) + p}{4 + 3(p+1)}\\
     &= \frac{ \Big(2p + 4 \Big) + p}{4 + 3(p+1)} = \frac{3p + 4 }{ 3p+7}\\
\end{align*}   

therefore, $ \lim_{p \longrightarrow \infty }\mathfrak{J}(D_{2^2p}) = \lim_{p \longrightarrow \infty } \frac{3p + 4 }{ 3p+7} = 1.$

\end{proof}

Similar results can be found for finite dicyclic groups as shown below:

\begin{prop}

Let $G$ be a finite dicyclic group, then

    \begin{enumerate}
         \item{If $G = C_{p} \rtimes C_{4}$, then $\lim_{p \longrightarrow \infty} \mathfrak{J}(G) = 1$; } 
         \item{If $G = C_{p} \rtimes Q_{2^{n}}$, then $\lim_{p \longrightarrow \infty} \mathfrak{J}(G) = 1$;} 
          \item{If $G = C_{p^2} \rtimes C_{4}$, then $\lim_{p \longrightarrow \infty} \mathfrak{J}(G) = 1$;} 
           \item{If $G = C_{p^2} \rtimes Q_{8}$, then $\lim_{p \longrightarrow \infty} \mathfrak{J}(G) = 1$;} 
          \item{If $G = C_{q} \rtimes (C_{p}\rtimes C_{4})$, then $\lim_{p \longrightarrow \infty} \mathfrak{J}(G) = \frac{q}{q + 1}$, \text{where $p$ and $q$ are primes.}} 
    \end{enumerate}

\end{prop}

\begin{proof}
The demonstration of this proposition follows as in the case of dihedral groups. 
The most considerable changes follow in the denominator of the expressions where $L(G) = \tau(2n) + \sigma(n)$ if $|G| = 4n$, by  [8], [9]. Using [5] it was possible to determine for the cases above a expression for $Nil(G)$:

\begin{itemize}
     \item{$Nil(C_{p} \rtimes C_{4}) = p +4$, for $p \geq 3$;}
     \item{$Nil(C_{p} \rtimes Q_{2^{n}}) = (2^{n-1} -1 )p + 2n$, for $p \geq 3$;}
     \item{$Nil(C_{p^2} \rtimes C_{4}) = p^2 + 6$, for $p \geq 3$;}
     \item{$Nil(C_{p^2} \rtimes Q_{8}) = 3p^2 + 9$, for $p \geq 3$;}
     \item{$Nil(C_{q} \rtimes (C_{p}\rtimes C_{4})) = qp + 8$, for $p \geq 3$.}
\end{itemize}

By making the necessary substitutions and calculating the limits, the desired results are obtained.
\end{proof}


\begin{ex}
Consider $n = 10$ in $C_{p} \rtimes Q_{2^{n}}$, then
\begin{align*}
    \lim_{p \longrightarrow \infty }\mathfrak{J}(C_{p} \rtimes Q_{1024}) = \lim_{p \longrightarrow \infty } \frac{511p + 20 }{ \tau(512p) + \sigma(256p)} = \lim_{p \longrightarrow \infty } \frac{511p + 20 }{ 531 + 511p} = 1.
\end{align*}
\end{ex}

\begin{ex}
Consider $q = 3$ in $C_{q} \rtimes (C_{p}\rtimes C_{4})$, then
\begin{align*}
    \lim_{p \longrightarrow \infty }\mathfrak{J}(C_{3} \rtimes (C_{p}\rtimes C_{4})) = \lim_{p \longrightarrow \infty } \frac{3p + 8 }{ \tau(6p) + \sigma(3p)} = \lim_{p \longrightarrow \infty } \frac{3p + 8 }{ 12 + 4p} = \frac{3}{4}.
\end{align*}

\end{ex}








In this sense we can consider another family of non-nilpotent supersoluble groups. In particular, we will consider a family that has the order configuration as $pq^{n}$, where p and q are primes with the following configuration: $C_{p} \rtimes C_{q^{n}}$.

\begin{theorem} \label{Theorem 2.10}
    
Let $G =C_{p} \rtimes C_{q^{n}} $, then for $n \geq 2$
  \[
\mathfrak{J}(C_{p} \rtimes C_{q^{n}}) = \frac{2(n + 1) + p - 2}{2(n + 1) + p - 1}.
\]
\end{theorem}

\begin{proof}
    For each divisor \( d \) of \( |G| = pq^n \), there exists at least one subgroup of order \( d \). In particular, the group \( G \) contains \( n+1 \) subgroups corresponding to the divisors of \( q^n \), as \( q^n \) has exactly \( n+1 \) divisors. Additionally, there is exactly one subgroup of order \( p \). Together, these subgroups account for a total of \( 2(n + 1) \) subgroups directly associated with the divisors of \( |G| \). However, the semidirect product structure of \( G \) introduces further complexity, resulting in additional subgroups. Specifically, for each subgroup \( H \) of \( C_p \) (a cyclic group), subgroups can be formed by combining \( H \) with subgroups of \( C_{q^n} \). Since \( C_p \) is cyclic, it has \( p - 1 \) proper nontrivial subgroups, along with the trivial subgroup and \( C_p \) itself. These \( p - 1 \) proper subgroups contribute additional subgroups of \( G \) beyond those accounted for by \( 2(n + 1) \). Thus, the total number of subgroups in \( G \) is given by \( 2(n + 1) + (p - 1) \), reflecting both the direct contributions from the divisors of \( |G| \) and the additional subgroups arising from the semidirect product structure.
\end{proof}

\begin{ex}
Take $n = q = 2$ and $p = 11$ in $C_{p} \rtimes C_{q^n}$, then
\begin{align*}
    \mathfrak{J}(C_{11} \rtimes C_{4}) = \frac{2(2 + 1) + 11 - 2}{2(2 + 1) + 11 - 1} = \frac{15}{16}.
\end{align*}
\end{ex}

\begin{ex}
Take $n = 4, q = 3$ and $p = 5$ in $C_{p} \rtimes C_{q^n}$, then
\begin{align*}
    \mathfrak{J}(C_{5} \rtimes C_{81}) = \frac{2(4 + 1) + 5 - 2}{2(4 + 1) + 5 - 1} = \frac{13}{14}.
\end{align*}
\end{ex}

In [10] the Cyclicity degrees were defined. This metric can be correlated with the function under study considering that cyclic subgroups will always be nilpotent, thus it is possible to define the order relationship:

\begin{equation*}
    cdeg(G) \leq \mathfrak{J}(G).
\end{equation*}

Note that a family of groups considered in Theorem \ref{Theorem 2.10} satisfies the conditions for a group whose degree of cyclicity is determined by Theorem 2.1 in [10]. With this, we can obtain the following result:

\begin{cor}
    Let $G =C_{p} \rtimes C_{q^{n}} $, then for $n \geq 2$, then

\begin{equation*}
        cdeg(G) = \mathfrak{J}(G). 
\end{equation*}

\end{cor}

The function apparently does not present a well-defined behavior in terms of group quotients, in this sense it is not possible to immediately define an order relationship of the type $\mathfrak{J}(G/H) \leq \mathfrak{J}( G) $ or $\mathfrak{J}(G) \leq \mathfrak{J}(G/H)$. But considering some specific subgroups it is possible to determine order relations in terms of $\mathfrak{J}$.

\begin{prop}

Let $G$ be a finite group, then

    \begin{enumerate}
         \item{$\mathfrak{J}(G) \leq \mathfrak{J}(Z(G))$}; 
         \item{$\mathfrak{J}(G) \leq \mathfrak{J}(\Phi(G))$}; 
         \item{$\mathfrak{J}(G) \leq \mathfrak{J}(F(G))$}. 
    \end{enumerate}

\end{prop}

where $Z(G)$ is the center of $G$, $\Phi(G)$ is the Frattini subgroup of $G$, and $F(G)$ is the Fitting subgroup of $G$. The proof immediately follows by recalling that these subgroups are always nilpotent. We will now present the main result of this work.

\begin{theorem}
 Let $D_{2^{m}p}$ be a subfamily of $D_{2n}$, where $m \geq 2$. Then the set $ \{ \prod_{n \in P} \mathfrak{J}(D_{2^{m}p_{n}})| P \subset \mathbb{N} \setminus \{ 0 \}, |P| < \infty, \, \text{and} \, \, p_{n} \, \text{is the nth prime number}, \forall n \in P \}$ is dense in $(0, 1]$.
\end{theorem}

\begin{proof}
  
In [6] the Lemma 4.1, states that for a given positive sequence $(x_{n})_{n \geq 1 }$ of real numbers such that $ \lim_{n \longrightarrow \infty} x_{n}=0$ and $ \sum_{n=1}^{\infty} x_{n}$ is divergent, so the set containing the sums of all subsequences of $(x_{n})$ is dense in $[0, \infty)$. Let's show that these two conditions are valid for the sequence $x_n = -\ln(\mathfrak{J}(D_{2^{m}p_{n}})), \forall n \geq 2$. Note that $x_{n} > 0, \forall n \geq 2$. This way we have:

\begin{align*}
   \lim_{n \longrightarrow \infty} x_{n} =   \lim_{n \longrightarrow \infty} -\ln(\mathfrak{J}(D_{2^{m}p_{n}})) &= \lim_{n \longrightarrow \infty} -\ln \left(    \frac{\tau(2^{m}p_{n}) \;+\; \sum_{\substack{2^r \mid 2^{m}p_{n} \\ r \geq 0}} \frac{2^{m}p_{n}}{2^r}}{\tau(2^{m}p_{n}) + \sigma(2^{m}p_{n})} \right)\\ 
   &= \lim_{n \longrightarrow \infty} -\ln \left(  \frac{2(m+1) + p_{n}(2^{m+1} - 1)}{2(m+1)+(2^{m+1} - 1)(1 + p_{n})}   \right)\\
   &=0.
\end{align*}

Now consider the series $\sum_{n}^{\infty} -\ln(\mathfrak{J}(D_{2^{m}p_{n}}))$. When \(p_n\) is large, the logarithm argument can be expanded in Taylor series around 1. Let \(x = \frac{1}{p_n}\), with \(x \to 0\) . So, it is possible to approximate:

\[
\frac{2(m+1) + p_n(2^{m+1} - 1)}{2(m+1) + (2^{m+1} - 1)(1 + p_n)} \approx 1 - \frac{2(m+1)}{p_n(2^{m+1} - 1)}.
\]

Using logarithm expansion for values close to 1 (\(\ln(1-x) \approx -x\)):

\[
-\ln \left( 1 - \frac{2(m+1)}{p_n(2^{m+1} - 1)} \right) \approx \frac{2(m+1)}{p_n(2^{m+1} - 1)}.
\]

Thus, the general term \(a_n\) behaves asymptotically as:
\[
a_n \sim \frac{2(m+1)}{p_n(2^{m+1} - 1)}.
\]

As prime numbers \(p_n\) grow approximately as \(p_n \sim n \ln(n)\). Therefore:
\[
a_n \sim \frac{2(m+1)}{(2^{m+1} - 1)n \ln(n)}.
\]

Now, we compare the given series with the reference series:

\[
\sum_{n=1}^\infty \frac{1}{n \ln(n)}.
\]

This reference series is divergent (see [1]). As \(a_n\) has the same asymptotic order as the divergent series, it is possible to conclude that the original series also diverges. By Lemma 4.1 in [6], the set of sums of all finite subsequences of \((x_n)\) is dense in \([0, \infty)\). In terms of the new sequence, this implies that

   \[
   \overline{\left\{-\ln\prod_{n \in P} \frac{2(m+1) + p_n(2^{m+1} - 1)}{2(m+1) + (2^{m+1} - 1)(1 + p_n)} \;\middle|\; P \subset \mathbb{N} \setminus \{0\}, |P| < \infty \right\}} = [0, \infty).
   \]

The exponential function \(e^{-x}\) is continuous, which preserves the density in the set. Then,

   \[
   \overline{\left\{\prod_{n \in P} \frac{2(m+1) + p_n(2^{m+1} - 1)}{2(m+1) + (2^{m+1} - 1)(1 + p_n)} \;\middle|\; P \subset \mathbb{N} \setminus \{0\}, |P| < \infty \right\}} = (0, 1].
   \]

   Follows,

   \[
    \prod_{n \in P} \frac{\tau(2^m p_n) + \sum_{\substack{2^r \mid 2^m p_n \\ r \geq 0}} \frac{2^m p_n}{2^r}}{\tau(2^m p_n) + \sigma(2^m p_n)} = \prod_{n \in P} \mathfrak{J}(D_{2^m p_n}),
   \]
   
follows the desired result.

\end{proof}
\newpage

\subsection{Further research}

\quad \, It would be interesting in future work to evaluate relationships between the function $\mathfrak{J}$ and other functions defined based on quotients, such as the degree of normality $(ndeg)$. Furthermore, it would be interesting to study relationships between the function $\mathfrak{J}$ and those presented in [4] and [6].

Furthermore, based on the function \(\mathfrak{J}(G)\), it was possible to formulate the following conjecture:

\begin{conjecture}
Let \( n \in \mathbb{N} \) be an even number with \( n \geq 2 \), so that \( n \) can be written as \( n = 2k \), where \( k \in \mathbb{N} \) and \( k \geq 1 \). We define the following arithmetic functions:
\begin{itemize}
\item \(\tau(n) = \sum_{d \mid n} 1\), the number of positive divisors of \( n \);
\item \(\sigma(n) = \sum_{d \mid n} d\), the sum of the positive divisors of \( n \); \item \(S_2(n) = \sum_{\substack{r \in \mathbb{N}_0 \\ 2^r \mid n}} \frac{n}{2^r}\), the sum of the successive divisions of \( n \) by the powers of 2 that divide it. \end{itemize}
With these functions, we define the random variable \(\mathfrak{J}: \mathbb{N} \rightarrow (0,1]\) by:

\[
\mathfrak{J}(D_n) = \frac{\tau(n) + S_2(n)}{\tau(n) + \sigma(n)}.
\]

Consider a random sample with replacement formed by the values \(\mathfrak{J}(D_2), \mathfrak{J}(D_4), \dots, \mathfrak{J}(D_n)\), with \( n \) sufficiently large. Let \(\mu\) be the mean and \(\sigma\) be the standard deviation of the distribution of \(\mathfrak{J}(D_n)\). We define the sample mean \(\bar{X}_n\) as:

\[
\bar{X}_n = \frac{1}{n} \sum_{i=1}^n X_i,
\]

where \(X_i\) are independent and identically distributed random variables with the same distribution as \(\mathfrak{J}(D_n)\). The normalization of the sample mean is given by:

\[
Z_n = \frac{\bar{X}_n - \mu}{\sigma / \sqrt{n}}.
\]

By the Central Limit Theorem, \(Z_n\) converges in distribution to a standard normal \(\mathcal{N}(0,1)\) when \( n \to \infty \), since the support of \(\mathfrak{J}(n)\) is contained in the finite interval \((0,1]\), ensuring that all its moments, including the variance, are finite. Thus, we conclude that the distribution of \(\mathfrak{J}(n)\) is asymptotically normal.

\end{conjecture}

Considering the presented conjecture, a robust simulation was performed using a data set that involves the proportion of the first 5 million orders of the dihedral groups, that is, from \(\mathfrak{J}(D_2)\) to \(\mathfrak{J}(D_{10\,000\,000})\). For this analysis, 1,000 random selections with replacement were performed, covering different subset sizes: 30; 500; 1,000; 100,000; 500,000; 1,000,000; 2,500,000 and 5,000,000 elements.

To evaluate the adherence of the sampling distribution of the variable \(\mathfrak{J}(n)\) to the standard normal distribution, illustrative graphs will be presented. These graphs consist of histograms with the superimposition of the theoretical normal density curve and Q-Q (Quantile-Quantile) plots. The histogram allows observing the shape of the distribution of the standardized sample means, while the theoretical normal density curve facilitates visual comparison with a standard normal distribution. In turn, the Q-Q plot compares the sample quantiles with the theoretical quantiles of the standard normal. The proximity of the points to the red line in the Q-Q plot indicates that the distribution of the standardized sample means is aligned with the normal distribution.

\begin{figure}[H]
    \centering
    \begin{minipage}[b]{0.47\linewidth}
        \centering
        \includegraphics[width=\linewidth]{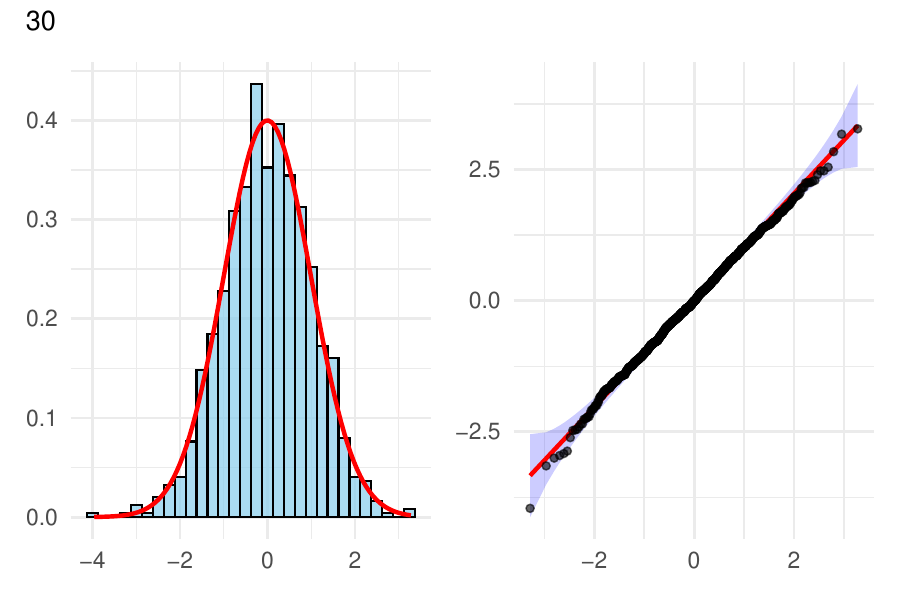}
        \caption{Histograms of simulated samples with overlay of the density curve of the standard normal distribution and Q-Q plot for sample size 30.}
        \label{fig:plot1}
    \end{minipage}
    \hfill
    \begin{minipage}[b]{0.47\linewidth}
        \centering
        \includegraphics[width=\linewidth]{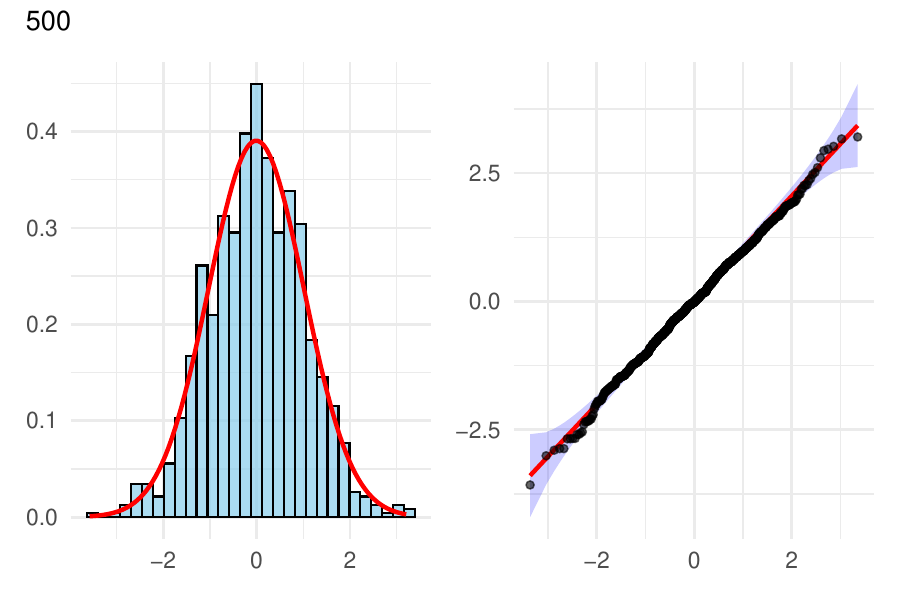}
        \caption{Histograms of simulated samples with overlay of the density curve of the standard normal distribution and Q-Q plot for sample size 500.}
        \label{fig:plot2}
    \end{minipage}
\end{figure}

\begin{figure}[H]
    \centering
    \begin{minipage}[b]{0.47\linewidth}
        \centering
        \includegraphics[width=\linewidth]{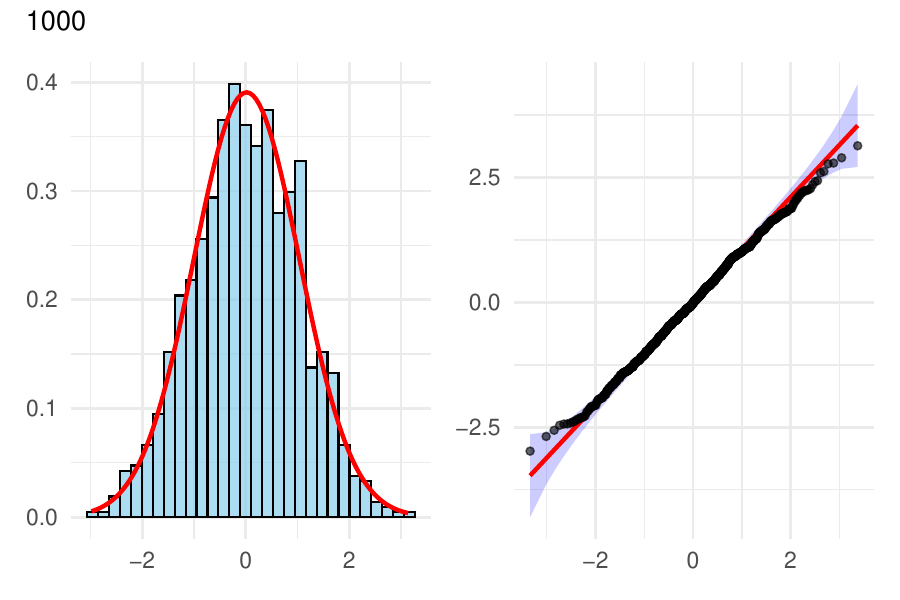}
        \caption{Histograms of simulated samples with overlay of the density curve of the standard normal distribution and Q-Q plot for sample size 1,000.}
        \label{fig:plot3}
    \end{minipage}
    \hfill
    \begin{minipage}[b]{0.47\linewidth}
        \centering
        \includegraphics[width=\linewidth]{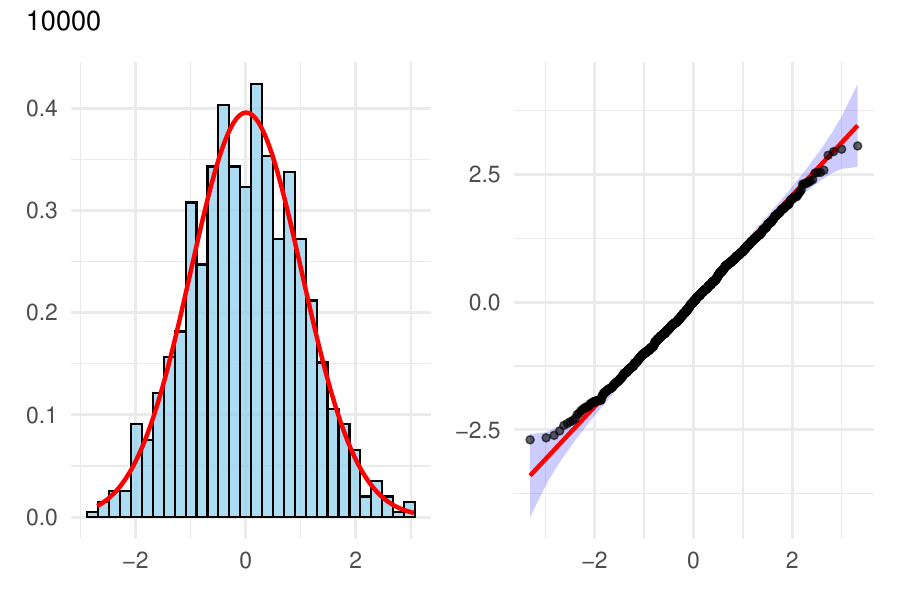}
        \caption{Histograms of simulated samples with overlay of the density curve of the standard normal distribution and Q-Q plot for sample size 10,000.}
        \label{fig:plot4}
    \end{minipage}
\end{figure}

\begin{figure}[H]
    \centering
    \begin{minipage}[b]{0.47\linewidth}
        \centering
        \includegraphics[width=\linewidth]{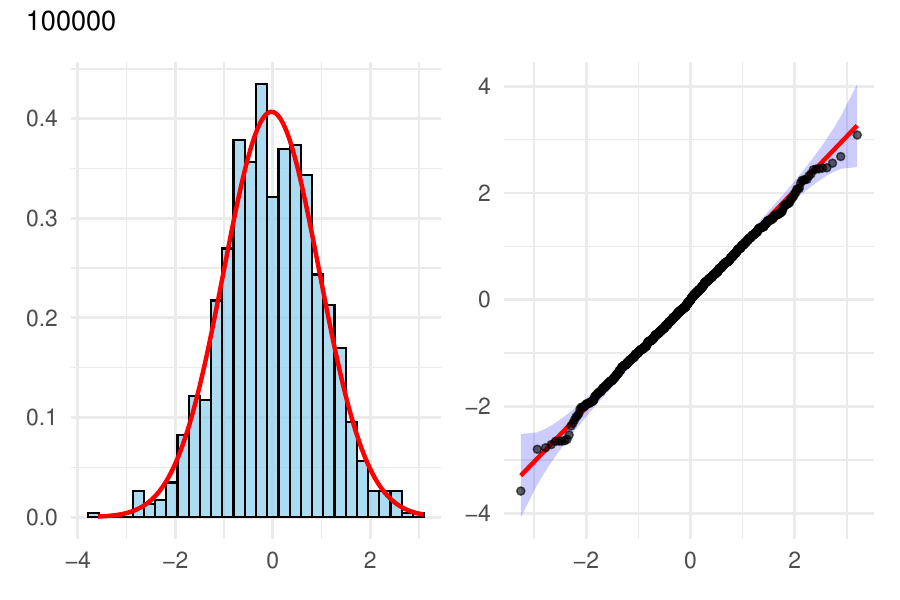}
        \caption{Histograms of simulated samples with overlay of the density curve of the standard normal distribution and Q-Q plot for sample size 100,000.}
        \label{fig:plot5}
    \end{minipage}
    \hfill
    \begin{minipage}[b]{0.47\linewidth}
        \centering
        \includegraphics[width=\linewidth]{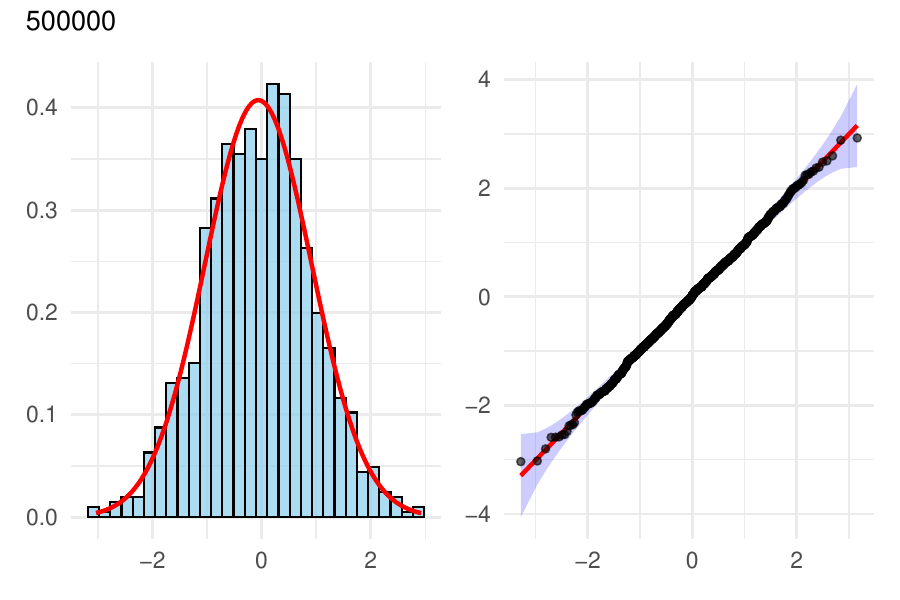}
        \caption{Histograms of simulated samples with overlay of the density curve of the standard normal distribution and Q-Q plot for sample size 500,000.}
        \label{fig:plot6}
    \end{minipage}
\end{figure}

\begin{figure}[H]
    \centering
    \begin{minipage}[b]{0.47\linewidth}
        \centering
        \includegraphics[width=\linewidth]{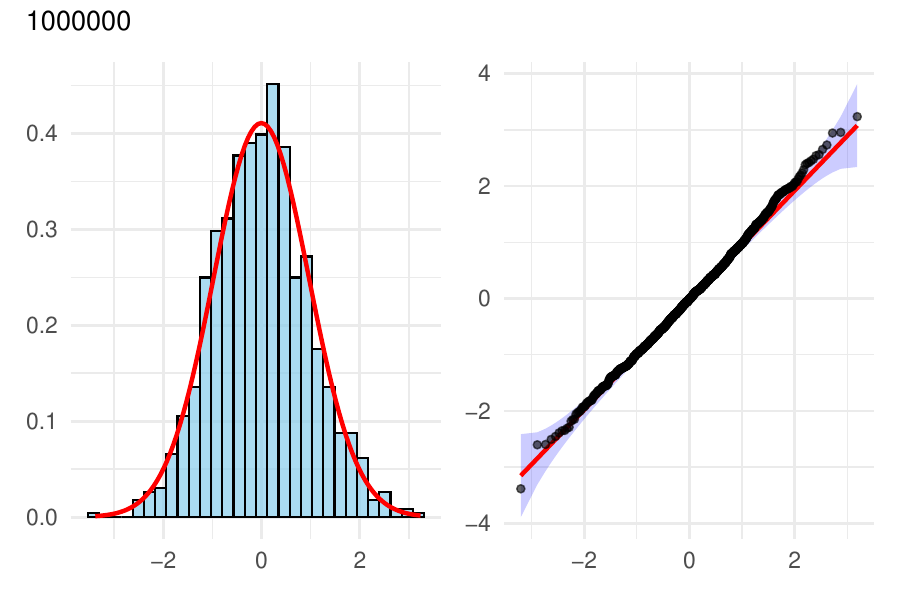}
        \caption{Histograms of simulated samples with overlay of the density curve of the standard normal distribution and Q-Q plot for sample size 1,000,000.}
        \label{fig:plot7}
    \end{minipage}
    \hfill
    \begin{minipage}[b]{0.47\linewidth}
        \centering
        \includegraphics[width=\linewidth]{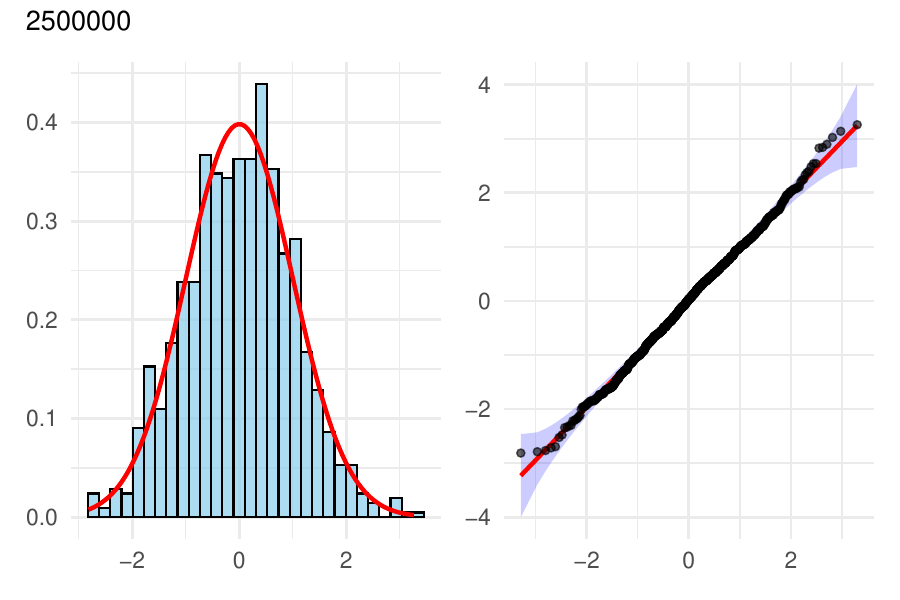}
        \caption{Histograms of simulated samples with overlay of the density curve of the standard normal distribution and Q-Q plot for sample size 2,500,000.}
        \label{fig:plot8}
    \end{minipage}
\end{figure}

Thus, it is possible to clearly observe the behavior predicted by the Central Limit Theorem when using a large number of samples, especially in samples with replacement whose size is equivalent to that of the simulated population, as illustrated in the next graph. This pattern becomes even more evident in this context, reinforcing the robustness and reliability of the results obtained. It is worth noting that the simulation of all values — covering up to 5 million elements — required a high computational cost and a significant processing time.

\begin{figure}[H]
    \centering
    \includegraphics[width=.8\linewidth]{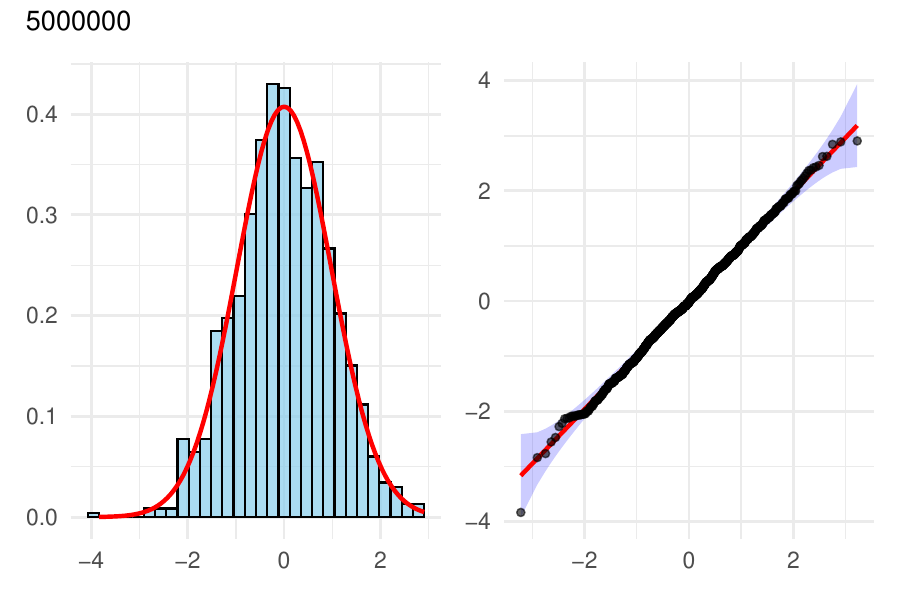}
    \caption{Histograms of simulated samples with overlay of the density curve of the standard normal distribution and Q-Q plot for sample size 5,000,000.}
    \label{fig:enter-label}
\end{figure}

To complement the visual analysis, the Kolmogorov-Smirnov (KS) goodness-of-fit test was applied to each sample size. This statistical test, introduced by Kolmogorov (1933) and later improved by Smirnov (1948), aims to verify whether a sample follows a specific theoretical distribution — in this case, the standard normal distribution \(\mathcal{N}(0,1)\) [3]. The test consists of comparing the empirical distribution function of the sample with the cumulative distribution function of the standard normal.

\(p\)-values greater than 0.05 (5\% significance level, equivalent to 95\% confidence) indicate that there is insufficient statistical evidence to reject the null hypothesis that the sample follows a normal distribution. This suggests that the data exhibit behavior compatible with normality. On the other hand, \(p\)-values lower than 0.05 indicate significant deviations from normality, indicating that the sample distribution differs significantly from the standard normal.

\begin{table}[h]
\centering
\caption{Kolmogorov-Smirnov test results for different sample sizes.}
\begin{tabular}{|c|c|c|}
\hline
\textbf{Sample Size} & \textbf{KS Statistic} & \textbf{p-value} \\ \hline
30                          & 0.014614              & 0.9832        \\ \hline
500                         & 0.019739              & 0.8307       \\ \hline
1,000                       & 0.035131              & 0.1694        \\ \hline
10,000                      & 0.023152              & 0.6573        \\ \hline
100,000                     & 0.027752              & 0.4244        \\ \hline
500,000                     & 0.035566              & 0.1593        \\ \hline
1,000,000                   & 0.02872              & 0.3815        \\ \hline
2,500,000                   & 0.023783              & 0.6236        \\ \hline
5,000,000                   & 0.018089              & 0.899        \\ \hline
\end{tabular}
\label{tab:ks_test}
\end{table}

This table complements the visual analysis by offering a quantitative assessment of the adherence of the sampling distributions to normality, strengthening the interpretation of the graphs presented. The results clearly confirm that all the data fit according to the proposed conjecture, further demonstrating the consistency of the observed patterns.

\subsection{Acknowledgments}

\quad \, We would like to express our gratitude to Professors Raimundo Bastos Jr., Roberto Vila Gabriel, Igor Lima, and Eduardo Antônio for their support and valuable suggestions for this work.

\newpage

\end{document}